%% file: main.tex
\documentclass{article}
\usepackage[utf8]{inputenc}
\usepackage{amsmath,amssymb,amsthm, xcolor, enumitem, comment, todonotes}
\usepackage{url}
\usepackage[all]{xy}

\theoremstyle{definition}
\newtheorem{definition}{Definition}

\newtheorem{question}[definition]{Question}

\newtheorem{notation}[definition]{Notation}

\newtheorem{remark}[definition]{Remark}

\theoremstyle{plain}
\newtheorem{theorem}[definition]{Theorem}
\newtheorem{proposition}[definition]{Proposition}
\newtheorem{lemma}[definition]{Lemma}
\newtheorem{corollary}[definition]{Corollary}

\input DominikMakros

\input OmerMakros

\title{Rudin-Keisler capturing and Mutual Stationairy at successors of Singulars}
\author{Dominik Adolf and Omer Ben-Neria\footnote{
The second author would like to thank the Israel Science Foundation (Grant 1832/19).}}
\date{\today}

\begin{document}

\maketitle
\begin{abstract}
    We introduce a combinatorial notion of measures called Rudin-Keisler capturing and use it to give a new construction of elementary substructures around singular cardinals. The new construction is used to establish mutual stationary results at the first  successor of singular cardinals $\la \aleph_{\omega n + 1}\ra_{n< \omega}$.
    %and answer a related question of Woodin. 
\end{abstract}
\section{Introduction}

The notion of Mutually Stationary sequences was introduced by Foreman and Magidor in \cite{mutstat}.
For an increasing sequence of regular cardinals $\la \kappa_n \mid n < \omega\ra$ and a sequence $\vec{S} = \la S_n \mid n < \omega\ra$ of subsets $S_n \subseteq \kappa_n$, we say $\vec{S}$ is mutually stationary sequence (denoted $MS(\vec{S})$) if every set theoretic algebra $\A = \la \kappa,\in,F_n\ra_n$ on $\kappa$ has a sub-algebra $X \elem \A$ such that $\sup(X \cap \kappa_n) \in S_n$ for all $n < \omega$.\footnote{The notion of mutually stationary sequences is also defined for longer sequences. The version for $\omega$ sequences is the most relevant for this work.}

The notion of mutually stationary sequences can be used to form a theory of stationarity at singular cardinals. 
It has also been connected to other notions in set theoretic algebras such as free sets and J\'{o}nsson Algebras. For example, assuming $\aleph_\omega$ is strong limit, then, by Silver, the existence of a sequence $\la k_n \mid n < \omega\ra$ of increasing finite ordinals $k_n < \omega$ such that the sequence $\la S_n = \omega_n \cap \cof(k_n)\ra_n$ is mutually stationary is equivalent to $\aleph_\omega$ being Jonsson (see \cite{Eskew2020-ESKOAS} for the case when $2^{\aleph_0} \geq \aleph_\omega$).
We refer the reader to \cite{ForemanSurvey} for a survey of this theory.
%Much work has been done around the properties of sequences $\vec{S}$ which are mutually stationary.

Foreman and Magidor proved (in ZFC) that every sequence $\la S_n\ra_n$ of stationary sets $S_n \subseteq \kappa_n \cap \cof(\omega)$ is mutually stationary, and raised the question about sequences consisting of ordinals of other cofinalities being mutually stationary. The 
theory of mutually stationary sets and its variation has been developed in a series of studies, including lower bounds and properties of mutually stationary sets in canonical inner models (Koepke and Welch \cite{KoeWel-MS}, Schindler \cite{schindler_2005},  the first author, Cox and Welch \cite{Adolf2018}, the first author \cite{altmutstatII}), ZFC constraints, and connection with cardinal arithmetic and set theoretic algebras  (Cummings Foreman and Magidor \cite{Cummings2006},
Liu and Shelah \cite{LiuShe-MS}, \cite{Shelah2021}), variations and extensions of mutually stationary sets (Apter \cite{Apter2005}, Cummings, Foreman and  Schimmerling \cite{Organic}, Chen and Neeman \cite{ChanNeeman}), and forcing results (Koepke \cite{Koe-MS}, Chen and Neeman \cite{ChanNeeman},
Adkisson and Sinapova \cite{SinapovaAdkisson}, 
the second author  \cite{singularstatI}, 
\cite{singularstatII}).

% and prove ZFC results and consistency results using methods in combinatorial set theory, cardinal arithmetic, inner model theory, and forcing with large cardinals. 
%The consistency results among all sequences $\la S_n\ra_n$ of stationary sets of a large variety of types, typically given by restriction of the cofinality of ordinals that can appear in the sets $S_n$. 
In \cite{singularstatI} it is shown that the existence of sufficiently strong ideals on the first uncountable cardinals, $\omega_n$, $n < \omega$, implies that for every $k < \omega$, every sequence $\la S_n \ra_{k < n < \omega}$ of stationary sets $S_n \subseteq \omega_n \cap \cof(\omega_k)$ is mutually stationary.
Such ideals on the $\omega_n$'s can be found after collapsing supercompact cardinals $\kappa_n$, $n < \omega$ to become the first uncountable cardinals. 
The formation of the ideals on the $\omega_n$'s in \cite{singularstatI} relies on the nice feature of the Levy-collapse forcing, and does not extend to successors of singular. 
Woodin asked\footnote{MAMLS meeting at VCU
April 1-2, 2017} if the same result is consistent when replacing the $\aleph_n$'s with the first successors of singular cardinals $\{ \aleph_{\omega\cdot n +1}\}_{n < \omega}$.
In this work we provide an affirmative answer. 
\begin{theorem}\label{Theorem:WoodinQuestion}
It is consistent relative to an $\omega+2$-strong cardinal that 
every sequence $\la S_n \ra_n$ of stationary sets $S_n \subseteq \aleph_{\omega \cdot n + 1}$ of bounded cofinality, is mutually stationary.  
\end{theorem}
The consistency strength stated in the theorem, of $\omega+2$-strong cardinal is not optimal. A better upper-bound will be stated following the method by which the theorem is established. 
The main ingredient in the proof of the theorem is a combinatorial principle of measures we call \emph{Rudin-Keisler capturing}. The relevant measures needed for the proof, can be placed in terms of strength slightly above the level of $o(\kappa) =  \kappa^{+(\omega + 1)}$. 
%We show it can be derived from sufficiently strong embeddings, and show how to apply it to establish mutual stationarity results in a Prikry forcing extensions. 
The Rudin-Kiesler capturing argument allows us to significantly extend both the possible sequences of cardinals on which mutually stationary sequences can appear, as well as the possible cofinalities witnessed in the stationary sets. 
We will prove the following theorem to illustrate this.
\begin{theorem}\label{Theorem:SecondMain}
It is consistent relative to an $\omega$-strong cardinal that there is an increasing sequence of cardinals $\la \kappa_n \mid n < \omega\ra$ with the following property:
For every  $\xi < \kappa_0$ and every sequence  $\la k_n \ra_n$  of finite ordinals, every sequence  $\la S^m_n \mid n < \omega, 1< m < k_n\ra$  of stationary sets 
$S_n^m \subseteq \kappa_n^{+m}  \cap \cof(\leq \xi)$ is mutually stationary. \\
Moreover, for a fixed sequence of $\<k_n\>$ this is possible to obtain in a model where the first uncountable cardinals $\{ \aleph_{\ell}\}_{1 \leq \ell < \omega}$ are the cardinals $\kappa_n^{+m}$, $n < \omega$, $2 \leq m \leq k_n$. 
\end{theorem}

In Section \ref{Section:Prelim} we review relevant preliminaries about structures, hulls, end-extensions, stationary sets, measures, generators, and Prikry forcing. 
In Section \ref{Section:RKcapturing} we introduce the notion of Rudin-Keisler capturing, establish its consistency and provide a first complete argument which shows how to use it to prove results about mutually stationary sets. 
In Section \ref{Section:Main} we extend the result from the previous section to prove the main theorem of this paper, from which we derive Theorems \ref{Theorem:WoodinQuestion}, \ref{Theorem:SecondMain} stated above. 
In Section \ref{Section:WeakRKcapturing} we describe a weakening of the Rudin-Keisler capturing,  which reduces the large cardinal assumption and provide additional flexibility in relevant constructions in the context of mutually stationary sets.
In Section \ref{Section:OpenProblem} we conclude this work with a list of open problems and further research directions.

\section{Preliminaries}\label{Section:Prelim}

\begin{definition}
A set $\mathcal{S} \subseteq \power(\lambda)$ is stationary if for every function $f : [\lambda]^{<\omega} \to \lambda$ there is $X \in \mathcal{S}$ that is closed under $f$. 
\end{definition}

Let $\theta > \lambda$ be a regular cardinal and let $\A = \<H_\theta; \in,<_\theta,\lambda\>$ where
$<_\theta$ is a fixed well-order of $H_\theta$. The following standard fact will be useful for us when arguing certain families to be stationary. 

%We will use a standard equivalent criterion for stationarity.
\begin{lemma}
    
A set $\mathcal{S} \subseteq \power(\lambda)$ is stationary if there is a  $X  \elem \A$ with $\mathcal{S} \in X$ and $X \cap \lambda \in \mathcal{S}$. 
\end{lemma}

\begin{definition}
Let $X \prec \mathfrak{A}$ and $A \subset \lambda$. 
The Skolem Hull of $A$ and $X$, denoted $\Sk^\mathfrak{A}(X \cup A)$ or $X_A$, is the set of values of the form $f(a)$, where $f:\left[\lambda\right]^{<\omega} \rightarrow H_\theta$ is a function in $X$ and $a \in [A]^{<\omega}$.
\end{definition}

We have $X \cup A \subseteq \Sk^\mathfrak{A}(X \cup A) \elem \A$.  Many of the constructions in this paper are obtained by taking hulls $\Sk^\mathfrak{A}(X \cup A)$ of sufficiently nice structures $X \elem \A$ by carefully chosen sets $A \subseteq \lambda$. 
We need a number of well-known basic facts about properties that are changed/unchanged when comparing $X$ with $X_A = \Sk^\mathfrak{A}(X \cup A)$. 

\begin{lemma}[Baumgartner]\label{thatonelemma}${}$
\begin{enumerate}
    \item Let $X \prec \mathfrak{A}$, and let $X_A := \Sk^\mathfrak{A}(X \cup A)$ where $A \subset \lambda$. If $\kappa \in X$ is a regular cardinal and $\sup(A) \leq \sup(\kappa\cap X)$ then $\sup(X \cap \kappa) = \sup(X_A \cap \kappa)$.
    
    \item Suppose that $X \prec \mathfrak{A}$ has $|X| = \aleph_n$ 
    and $X \cap \omega_{n+1} \in \omega_{n+1}$ for some $n < \omega$. 
    Let $\<S_k:k < m\>$ be a finite sequence of stationary sets $S_k \subseteq \omega_{n + k + 1} \cap \cof(\kleinergleich\aleph_n)$. Then there is some $A \subset \omega_{n + m}$ of size at most $\aleph_n$ such that $\sup(X_A \cap \aleph_{n + k + 1}) \in S_k$ for all $k < m$.
\end{enumerate}
\end{lemma}

%This lemma allows us to plug some finite size holes that might arise from the use of our other methods. This lemma works fine if we were only interested in sequences of stationary subsets of the $\aleph_n$'s. But if we are interested in what happens to cardinals outside that region we run into a problem. 
The first part is straightforward. The second builds on the first to form a descending inductive construction. 
Baumgartner's lemma works for regular cardinals that are already inside the structure. Increasing the structure as in the lemma will generally  add new regular cardinals. Using \emph{internally-approchable} structures is a useful way to obtain control of the cofinality of the new hull, below the added regular cardinals.

\begin{definition}\label{Def:IA}
  Let $\delta$ be a regular uncountable cardinal. A substructure $X \prec \mathfrak{A}$ is \emph{internally approachable} of length $\delta$ iff there exists a sequence $\<X_\alpha : \alpha < \delta\>$ of substructures $X_\alpha \elem \mathfrak{A}$ such that $\bigcup\limits_{\alpha < \delta} X_\alpha = X$ and 
        $\<X_\alpha: \alpha < \gamma\> \in X_\beta$ for every $\gamma < \beta < \delta$.
\end{definition}

For every $X$ as in the definition, 
$\cof(X \cap \kappa) = \delta$ for all regular $\kappa \in X$ above $\card(X)$ as witnessed by $\<\sup(X_\alpha \cap \kappa): \alpha < \delta\>$. 
A similar statement holds for Skolem hulls of internally approachable structures.

\begin{lemma}\label{IASkolemhulls}
    Let $\delta$ be a regular uncountable cardinal and $X \prec \mathfrak{A}$ be an internally approachable structure of length $\delta$.
    %witnessed by $\<X_\alpha: \alpha < \delta\>$, 
    For every $A \subset \lambda$ of size at most $\card(X)$,  $\cof(X_A \cap \kappa) = \delta$ for all regular $\kappa \in X_A$ above $\sup(A)$.
\end{lemma}

\begin{proof}
     Let $a$ be a finite subset of $A$ such that $\kappa = f(a)$ for some $f \in X$. It is easy to see that $X(a) := \Sk^\mathfrak{A}(X \cup a)$ is internally approachable of length $\delta$ as witnessed by $\<\Sk^\mathfrak{A}(X_\alpha \cup a): \alpha < \delta\>$. Therefore, $\cof(\sup(X(a) \cap \kappa)) = \delta$.  
     Since $X_A = \Sk^\mathfrak{A}(X(a) \cup A)$ then $\sup(X_A \cap \kappa) = \sup(X(a) \cap \kappa)$ by Lemma \ref{thatonelemma}.
\end{proof}

%This property motivates the next definition.
%\begin{definition}
 %   $Y \prec \mathfrak{A}$ is close to internally approachable (of length $\delta$) iff $Y = \Sk^\mathfrak{A}(X \cup A)$ where $X \prec \mathfrak{A}$ is internally approachable (of length $\delta$) and $A \subset \lambda$ is of size at most $\card(X)$.
%\end{definition}

%Note that being close to internally approachable is preserved under taking additional small Skolem hulls. Therefore Lemma \ref{IASkolemhulls} applies just as well to them.

%To extend Lemma \ref{thatonelemma} to sequences farther up than the $\aleph_n$'s. We use a localized version of internal approachability.
%\begin{definition} 
 %   Let $\delta$ be a regular uncountable cardinal, let $\zeta < \lambda$ be inaccessible. Let $X \prec \mathfrak{A}$ be o size less than $\zeta$. $X$ is internally approachable of length $\delta$ \emph{below} $\zeta$ iff there exists an internally-approachable  sequence of structures $\<X_\alpha: \alpha < \delta\>$ as in Definition \ref{Def:IA} such that        $\bigcup\limits_{\alpha < \delta} (X_\alpha \cap H_{\zeta}) = X \cap H_\zeta$.\end{definition}\todo{If we skip the proof of the next lemma, does this still justify inclusion?}

Combining the previous lemmas, we get the following version of Baumgartner's result.
\begin{lemma}\label{finitevariationII}
Let $\kappa_0$ be a cardinal and $\zeta_0 < \ldots < \zeta_n$ be a finite set of inaccessible cardinals above $\kappa_0$.  
   Suppose that $X \prec \mathfrak{A}$ is a substructure of size $\kappa_0$ that is internally approachable of length $\delta \leq  \kappa_0$.
    Then
   for every finite sequence $\<S^l_k :l < n, k < m_l\>$ of stationary sets $S^l_k \subset (\zeta_l)^{+(\tau^l_k)} \cap \cof(\eta^l_k)$ for ordinals $\<\tau^l_k: k < m_l\>$ and regular cardinals $\<\eta^l_k: k  < m_l\>$ below $\kappa_0$ there is  $A \subset (\zeta_{n - 1})^{+m}$ of size at most $\kappa_0$ such that \begin{enumerate}
       \item $\sup(X_A \cap (\zeta_l)^{+\tau^l_k}) \in S^l_k$ for all $l < n$ and all $k < m$,
       \item $\cof(X_A \cap (\zeta_l)^{+\tau'}) \in \{\eta^l_{m_l - 1},\ldots,\eta^l_{k + 1},\delta\}$ for $\tau' \in \left(\tau^l_{k - 1},\tau^l_k\right)$ and all $l < n$ and $k < m_l$ ($\tau^l_{-1} = 1)$, 
       \item  $\cof(\sup(X_A \cap \kappa)) = \delta$ for all regular $\kappa \in X_A \cap \zeta_n$ not covered by the above.
   \end{enumerate}
\end{lemma}

%\todo{updated this to be more in line with later parts}

%So far we have only concerned ourselves how to preserve the intersection of our structures with cardinals above the point where we insert additional ordinals. This was sufficient as long as we are only interested in finitely many points as this allows us to move from top to bottom. But if we want to be able to handle infinitely many points we need to move from bottom to top, so we need some method to preserve intersections at points below the point of insertion.

\begin{definition}
    Let $X \prec \mathfrak{A}$. A set of ordinals $A$ \emph{end-extends} $X$ below a cardinal $\kappa$ iff $\Sk^\mathfrak{A}(X \cup A) \cap \kappa = X \cap \kappa$.
\end{definition}

%The following lemma is essentially from (ADD REFERENCE!).
The following lemma is useful when dealing with end-extensions with changing domains. See \cite{LiuShe-MS} for a proof.
\begin{lemma}\label{dblendext} 
    Let $X \prec \mathfrak{A}$ and $A,B \subset \lambda$ such that $A$ end-extends $X$ below $\kappa$ where $\kappa^{\sup(B)} \leq \kappa$. Then $A$ end-extends $X_B$ below $\kappa$. 
\end{lemma}

We proceed to review material on measurable cardinals and the Rudin-Keisler ordering which will be used in constructing various end-extensions to substructures $X \elem \A$. 
Unlike the previous lemmas that required no special assumption from the background universe of set theory $V$, finding ordinals that end-extend structures is connected to measures on large cardinals.

\begin{definition}(Measures and Rudin-Keisler reductions)
\begin{enumerate}
    \item A measure $U$ on $[\kappa]^n$ is a non-principal $\kappa$-complete ultrafilter on $[\kappa]^n$, for some $n < \omega$.
$U$ is normal if $n = 1$ and $U$ is closed under diagonal intersections. 

\item A measure $U$ on $[\kappa]^n$ is \emph{Rudin-Keisler} reducible to a measure $V$ on $[\kappa]^m$ (written $U \leq_{RK} V$) iff there is some $f:\left[\kappa\right]^m \rightarrow \left[\kappa\right]^n$ such that $A \in U$ if and only if $f^{-1}(A) \in V$.
\end{enumerate}    
\end{definition}

Let $U$ be a measure on $[\kappa]^n$ for some $n < \omega$. We say that a property $\phi(x)$ holds for $U$-almost all $b \in [\kappa]^n$ when 
$\{ b \in [\kappa]^n \mid \phi(b) \} \in U$. 

Given an elementary embedding $j : V \to M$ with critical point $\cp(j) = \kappa$, we associate to every finite set of ordinals $a \in [j(\kappa)]^n$  a measure $U^j_a := \{A \subseteq \kappa\mid a \in j(A)\}$ on $[\kappa]^{|a|}$.

%Rudin-Keisler reduction.

\begin{definition}\label{Def:Generators}(Generators)
Throughout this work we assume $V$ satisfies GCH.
Let $j:V \rightarrow M$ be an elementary embedding with $\kappa = \cp(j)$.
\begin{enumerate}

    \item  Let $s \in [j(\kappa)]^{<\omega}$. An ordinal $\gamma$ is an $s$-generator (of $j$) iff $j(f)(s,a) \neq \alpha$ for every $a \in [\alpha]^{<\omega}$ and $f: [\kappa]^{<\omega} \rightarrow \kappa$.
    
    \item   A finite set of ordinals $a$, listed in descending order as $\<\xi_0,\ldots,\xi_{n - 1}\>$, is a \emph{generator} of $j$ iff for every $i < n$,  $\xi_i$ is a $\<\xi_0,\ldots,\xi_{i - 1}\>$-generator of $j$.
    
    \item A measure $U$ on $[\kappa]^n$ is \emph{minimally generated} if $a = [id]_{U}$ is a generator of the ultrapower embedding $\iota_U : V \to M_U \cong Ult(V,U)$. 
    \end{enumerate}
\end{definition}

\begin{notation}
For a finite sequence $s = \la s(i) : i < k\ra$ and $m \leq k$, we write $s \downharpoonright m$ to denote the tail $  \la s(i) : m \leq i < k\ra$ of $s$.
\end{notation}

\begin{lemma}\label{genclosed}
   Let $j:V \rightarrow M$ be an elementary embedding, 
   $\gamma \geq \kappa^+$ an $M$-cardinal, and $s \in [j(\kappa)]^{<\omega}$. 
   For every $\eta \in \left[\gamma,(\gamma^+)^M\right)$, $\eta$ is an $s$-generator of $j$ if and only if for every $f:\finsubsets{\kappa} \rightarrow \kappa$ and $a \in \finsubsets{\eta}$ either $j(f)(a \cup s) \geq (\gamma^+)^M$ or $j(f)(a \cup s) < \eta$.
\end{lemma}

\begin{proof}
    Assume not. For some $f:\finsubsets{\kappa} \rightarrow \kappa$ and $a \in \finsubsets{\eta}$ we have $j(f)(a \cup s) \in \left(\eta,(\gamma^+)^M\right)$. Note that $j(f)(a \cup s)$ cannot be equal to $\eta$ as it is an $s$-generator.
    
    Pick some function $g:\kappa^2 \rightarrow \kappa$ such that $\ptwimg{g}{(\{\alpha\} \times \card(\alpha))} = \alpha$ for all $\alpha < \kappa$.
    
    In $M$ then $j(g) \restr (\{j(f)(a \cup s)\} \times \gamma)$ is onto $j(f)(a \cup s)$. Therefore $\eta = j(g)(j(f)(a \cup s),\zeta)$ for some $\zeta < \gamma$.
    
    Define now a function $h: \finsubsets{\kappa} \times \kappa \rightarrow \kappa$ by $h(b,\beta) := g(f(b),\beta)$. Thus $\eta = j(h)(a \cup s,\zeta)$. Contradiction.
\end{proof}

\begin{corollary}\label{genclosedrem}${}$
    \begin{itemize}
        \item[$(a)$] The set of generators in $\left(\gamma,(\gamma^+)^M\right)$ is closed below $(\gamma^+)^M$.
        \item[$(b)$] If $(\gamma^+)^M = (\gamma^+)^V$ then the set of generators unbounded in $(\gamma^+)^M$.
        \item[$(c)$] $\eta < \min(s)$ is an $s$-generator iff for every  $f:\finsubsets{\kappa} \rightarrow \kappa$ there are $U^j_{ \{\eta\} \cup s}$-measure one many $b \in \left[\kappa\right]^{\length(s) + 1}$ such that for every $a \in [b(0)]^{<\omega}$,
        $f(a \cup (b\downharpoonright\length(s)))$ is either below $b(0)$ or at least $b(0)^+$.
    \end{itemize}\end{corollary}

\begin{definition}\label{Def:P_U}
Let $U$ be a measure on $\kappa$. 
Prikry forcing $\po_U$ consists of pairs $p = \la s, T\ra$ where $s \in [\kappa]^{<\omega}$ is an increasing sequence, and $T \subseteq  [\kappa]^{<\omega}$ is a tree with $stem(T) = s$, and for every $t \in T$ which extends $s$, the splitting set $\suc_T(t) = \{ \nu < \kappa \mid t \fr \la \nu\ra \in T\}$ belongs to $U$.
A condition $p' = \la s',T'\ra$ extends $p$ if $s' \in T$ and $T' \subseteq T_{>s'} = \{ r \in T \mid r \text{ is a proper end extension of } s'\}$. 
We further say that $p'$ is a direct extension of $p$ if $s' = s$.
\end{definition}

The Prikry Lemma asserts that every statement $\sigma$ in the forcing language can be decided by a direct extension. Since the direct extension ordering is $\kappa$-closed, it follows that $\po_U$ does not add bounded subsets to $\kappa$. Since $\po_U$ is $\kappa$-centered, it preserves all cardinals $\geq \kappa^+$.

A $V$-generic filter $G \subseteq \po_U$ is associated with its generating $\omega$-sequence $\vec{\kappa}= \bigcup\{ s \mid \exists T (s,T) \in G\}$, which is cofinal in $\kappa$.
We denote the sequence by $\vec{\kappa} =\la \kappa_n \mid n < \omega\ra$. 
%Many results about $\po_U$ go through the Strong Prikry Lemma.

\begin{lemma}(Strong Prikry Lemma)
For every condition $p = \la s,T\ra \in \po_U$ and a dense open set $D$ of $\po_U$ there is a direct extension $p^* = \la s,T^*\ra$ and some $n < \omega$ such that for every $t \in T^*$ of length $t = |n|$, $p \fr t := \la  t, T_{>t} \ra$ belongs to $D$.
\end{lemma}

Combining the the strong Prikry lemma with the closure property of the direct extension ordering gives the following useful result.

\begin{corollary}\label{Cor:PrikryNameGuessing}
Suppose that $U$ is a normal measure on $\kappa$, $\name{x} = \la \name{x}_n \mid n < \omega\ra$ is a $\po_U$-name, and $p = \la s,T\ra$ a condition which forces that each $\name{x}_n$ belongs to $V_{\name{\kappa}_{n+1}}$. Then there is a direct extension $p^* = \la s,T^*\ra$ of $p$, and a function $f_x : T^* \to [V_\kappa]^{\omega}$, so that for each $t \in T^*$, 
$f_x(t) \in V_{\max{t}}^{|t|}$ and 
$p \fr t \Vdash \check{f}_x(t) = \name{x} \uhr (|t|)$. 
\end{corollary}

\section{Rudin-Keisler Capturing}\label{Section:RKcapturing}

\begin{definition}($U$-sequences of stationary tuples)

    \begin{itemize}
        \item[$(a)$] A \emph{stationary-tuple} over $\kappa$ is a finite sequence $\<S_i: i < n\>$ where each $S_i$ is a stationary subset of a regular cardinal $\kappa^{\tau_i} \geq \kappa$ so that $\la \kappa^{+\tau_i} \mid i < n\ra$ forms  a strictly increasing sequence.
        
        We say that the sequence is of \emph{type} $\<(\tau_0,\eta_0),\ldots,(\tau_{n-1},\eta_{n-1})\>$ if moreover, for each $i < n$,  $S_i \subseteq \kappa_i \cap \cof(\eta_i)$, with both $\eta_i,\tau_i < \kappa$. 
        
        \item[$(b)$]     Let $U$ be a normal ultrafilter on $\kappa$. A function $f:\kappa \rightarrow H_\kappa$ is a $U$-sequence of stationary-tuples of type $\<(\tau_0,\eta_0),\ldots,(\tau_{n-1},\eta_{n-1})\>$ iff for $U$-almost all $\alpha$, 
        $f(\alpha) = \<S^\alpha_0,\ldots,S^\alpha_{n-1}\>$ is a stationary-tuple over $\alpha$ of type $\<(\tau_0,\eta_0),\ldots,(\tau_{n-1},\eta_{n-1})\>$.
        
        %$S^\alpha_i \subseteq \alpha^{+(i + 2)}$. 
        
        \item[$(c)$] A $U$-sequence $f$ of stationary-tuples of type $\<(\tau_0,\eta_0),\ldots,(\tau_n,\eta_n)\>$ is trivial iff 
        $f(\alpha) = \<\cof(\eta_i) \cap \alpha^{+\tau_i}\>$ for $U$-almost all $\alpha$.
    \end{itemize} 
\end{definition}

We will mainly focus on types where either $\tau_i = i + 2$, or $\tau_i = \omega +1$.

\begin{definition}(RK capturing of $U$-sequences)
%    Let $U$ be a normal ultrafilter on $\kappa$.
    \begin{itemize}
        \item[$(a)$] Let $f$ be a $U$-sequence of stationary tuples. We say $f$ is \emph{Rudin-Keisler captured} (relative to $U$) iff there exists a minimally generated ultrafilter $U \leq_{RK} V$ such that $a \in \prod \iota_V(f)(\kappa)$ where $a = [id]_V$ is the principal element of $V$.
        
        \item[$(b)$] Let $\tau < \kappa$ be an ordinal. We say that the full Rudin-Keisler capturing holds up to $\tau$ (relative to $U$) iff for every finite sequence 
        $$t = \la (\tau_0,\eta_0),\dots,(\tau_{n-1},\eta_{n-1}) \ra$$ with $\tau_0 > 1$ and $\tau_{n - 1} \leq \tau$, all $U$-sequences of stationary sets of type $t$ are Rudin-Keisler captured (relative to $U$).
       % \item[$(c)]$ Full singular Rudin-Keisler capturing holds (relative to $U$) iff for all $\eta \in \kappa \cap \reg$ all $U$-sequences of singular type $\eta$ are Rudin-Keisler captured.
    \end{itemize}
\end{definition}

%It is time to address the question of consistency.

\begin{lemma}\label{RKcaptcons}
   Let $j:V \rightarrow M$ be an elementary embedding of the universe with critical point $\kappa$ such that $\Pot(\kappa^{+\tau}) \subseteq M$. Then full Rudin-Keisler capturing holds up to $\tau$ relative to $U^j_\kappa$.
\end{lemma}

\begin{proof}
    First let $U := U^j_\kappa$, the normal measure associated to $j$. Let $f$ be any $U$-sequence of stationary tuples of type $t = \la (\tau_0,\eta_0),\dots,(\tau_{n-1},\eta_{n-1}) \ra$ with $\tau_0 > 1$ and $\tau_{n - 1} \leq \tau$. Consider $j(f)(\kappa) = \<S_0,\ldots,S_n\>$. We will find a generator $a$ of $j$ such that $a \in \prod j(f)(\kappa)$. Construction is by induction (going downwards).
    
    Assume that we have already found a generator $a_k$ such that $a_k \in S_k \times \ldots \times S_n$. (Note we start with $a_{(n + 1)} = \emptyset$ which is not a generator but on the other hand we do not require anything of it at this point.) We have that $(\kappa^{+\tau_{k-1}})^M = \kappa^{+\tau_{k-1}}$ is a successor cardinal in $M$. Therefore the set of $a_k$-generators of $j$ forms a club in $\kappa^{+\tau_{k-1}}$. By assumption this set is in $M$ and is, in fact, a club there. By \L o\'{s}'s Theorem, $M$ believes that $S_{k - 1}$ is a stationary set. So there must be some $\alpha_{k - 1}$ in $S_{k - 1}$ that is an $a_k$-generator. Then $a_{k - 1} = a_k \cup \{\alpha_{k - 1}\}$ is as wanted.
    
    Having constructed $a$ it is then easy to see that $U^j_a \geq_{RK} U^j_\kappa$ witnesses Rudin-Keisler capturing for $f$.
\end{proof}

\begin{remark}
    %\begin{itemize}
     %   \item[$(a)$] If we are merely interested in showing that sequences of types $\vec{\eta} \in \left[\kappa\right]^{(n + 1)}$ are captured then we only require that $\Pot(\kappa^{+(n + 2)}) \subset M$.
      %  \item[$(b)$]
       To capture the trivial sequence of type $\<(\tau_0+1,\eta_0),\ldots,(\tau_{n-1}+1,\eta_{n-1})\>$ we merely need an embedding $j:V \rightarrow M$ with a generator $a$ such that $a(i) \in \left((\kappa^{+\tau_i })^M,(\kappa^{+(\tau_i+1)})^M\right)$ of cofinality $\eta_i$.
    %\end{itemize}
\end{remark}

We will demonstrate how the Rudin-Keisler capturing implies Mutual Stationarity for sequences in Prikry generic extensions.

\begin{proposition}\label{Prop:CapturingMain}
    Let $U$ be a normal measure on $\kappa$ and $\vec{\kappa} = \la \kappa_n \mid n < \omega\ra$ be a $\po_U$-generic sequence.
    \begin{itemize}
        \item[$(a)$] If the trivial sequence of type $\<(2,\eta),\ldots,((k + 2),\eta)\>$ is Rudin-Keisler captured relative to $U$, then in $V\left[\vec{\kappa}\right]$ for every $\zeta < \kappa$ there is a stationary set of $x \subset \kappa$ such that 
        \begin{itemize}
            \item $\zeta \subset X$,
            \item $\cof(X \cap \kappa^{+(l + 1)}_n) = \eta$ for every $n < \omega$ and $1 
            \leq l \leq k$,
            \item $\cof(X \cap \kappa^*) = \zeta$ if $\kappa^* > \zeta$ is a regular cardinal in $X$ not covered by the above.
        \end{itemize}
        \item[$(b)$] If the full Rudin-Keisler capturing holds to some $\tau < \kappa$, then in $V\left[\vec{\kappa}\right]$ every function $S= \la \vec{S}(n) \mid n < \omega\ra$ of finite sequences  $S(n)  = \la S_n^i \mid i \leq \ell_n\ra$ of types $\la (\tau^n_0,\eta^n_0), \dots, (\tau^n_{\ell_n},\eta^n_{\ell_n})\ra$ with $1 < \tau^n_0 < \tau^n_{\ell_n} \leq \tau$ and $\<\eta^n_i: i \leq \ell_n, n < \omega\>$ bounded by some $\eta < \kappa$ is mutually stationary.
    \end{itemize}
\end{proposition}

    We give the proof for the second, more involved part (b). 
    Let $\name{S}$,$\la \name{\kappa}_n \mid n < \omega\ra$ be $\po_U$ for the function $S$ and the Prikry sequence, and $p = \la t,T\ra \in \po_U$ be a condition which forces the relevant statements about them. 
    We assume for notational simplicity that $t = \emptyset$.
    The general case follows from the same construction applied to the appropriate part of the Prikry tree $T$ (i.e., above $\kappa_{|t|}$) with the use
    of Lemma \ref{finitevariationII} to secure that the final addition of $t$ does not affect the result.
    By Corollary \ref{Cor:PrikryNameGuessing} we may assume there is a function  $F: T \rightarrow V_\kappa$ such that $(s,T_s) \forces \name{S}(|s|) = \check{F}(s)$ whenever $s \in T$. 
    Define for every $s \in T$ a $U$-sequence of stationary tuples $g_s$ of type $\<(\tau^s_i,\eta^s_i): i \leq \ell_{|s|})\>$ by $g_s(\alpha) := F(s \fr \la \alpha\ra)$.\footnote{In case $(a)$ $g_s$ is the trivial sequence of type $\<(2,\eta),\ldots,((k + 2),\eta)\>$.} 
     We will have established the Proposition when we prove the following technical lemma.

\begin{lemma}\label{Lemma:MainLemmaA}
    Assume that every $g_s$ is Rudin-Keisler captured by some minimally generated $U_s \geq_{RK} U$. 
    For every condition $(t,T)$, a $\po_U$-name of an algebra $\name{\mathfrak{A}}$, and a regular cardinal $\zeta$, $\eta \leq \zeta < \kappa$, there exists a direct extension $(t,T^*)$ of $(t,T)$ which forces the following: ``there is $X \elem \name{\mathfrak{A}}$ such that 
    \begin{itemize}
        \item[(i)] $\zeta \subset X$,
        \item[(ii)] for every $m < \omega$ if $\name{\kappa}_m > \zeta$ then $\sup(X \cap \dot{\kappa}_m^{+\dot{\tau}^m_i}) \in \name{S}^i_m$  for all $i \leq \ell_m$,
        \item[(iii)] $\cof(X \cap \kappa^{+\tau^*}_m) \in \{\zeta,\dot{\eta}^m_{\ell_m},\ldots,\dot{\eta}^m_{i + 1}\}$ for all $N \leq m < \omega$, all $i \leq \ell_n$ and all $\dot{\tau}^m_{i - 1} < \tau^* < \dot{\tau}^m_i$ ($\dot{\tau}^m_{-1} = 1$),
        \item[(iv)] $\cof(X \cap \kappa^*) = \zeta$ for all regular cardinals $\kappa^* > \zeta$ in $X$ not covered by the above.
    \end{itemize}
\end{lemma}

\begin{proof}
Fix in $V$ a canonical sequence of functions $\la h^\mu \mid \mu < \kappa^+\ra \in {}^\kappa \kappa$. Therefore $\mu = j_{U}(h^\mu)(\kappa)$ for all $\mu < \kappa^+$.
    We begin by fixing $U_s$ as in the statement of the lemma. Each $U_s$ concentrates on finite subsets of $\kappa$ of fixed size $\ell_s + 1$. Let $\pi_s:\left[\kappa\right]^{\ell_s + 1} \rightarrow \kappa$ witness $U_s \geq_{RK} U$.
    
    In the generic extension $V[\vec{\kappa}]$, let $U_n : =  U_{\vec{\kappa}\uhr n}$, $\ell_n := \ell_{\vec{\kappa} \restr n}$, $\tau^n_i = \tau^{\vec{\kappa} \restr n}_i$, and $\eta^n_i = \eta^{\vec{\kappa}\restr n}_i$. 
    The structure $X$ as in the statement of the lemma will be an elementary substructure of 
    $\A = \<H_\theta^{V[\vec{\kappa}]}; \in,<_\theta,\lambda\>$ for $\theta = (2^\kappa)^{++}$ and $\lambda = (2^\kappa)^+$. We will make use the basic properties and terminology from Section \ref{Section:Prelim} for the construction. 
    Back in $V$, let $\name{\mathfrak{A}}$,  $\name{\vec{\kappa}} = \<\dot{\kappa}_n \mid n < \omega\>$, and 
    $\<\name{U}_n \mid n < \omega\>$
    be names for these objects. 
    
    Fixing a regular cardinal $\zeta < \kappa$, we will define interdependent sequences of trees $\<T^\alpha \mid \alpha < \zeta\>$, and of $\po_U$-names $\<\dot{Y}^\alpha \mid \alpha < \zeta\>$ of substuctures of $\A$.
    
    The sequence $\<\dot{Y}^\alpha \mid \alpha < \zeta\>$ will be (forced to be) internally approachable, and each structure $\name{Y}^\alpha$ will be of size smaller than $\zeta < \kappa_N$. The first structure
    $\dot{Y}^0$ is taken to be $\dot{Y}^0 = \Sk^{\dot{\mathfrak{A}}}(\{\dot{U}_n,\dot{\kappa}_n: n < \omega\})$. In particular, each $\name{Y}^\alpha$ is forced to contain the generic Prikry sequence $\vec{\kappa}$. 
    In the course of the construction, we will make use of $\po_U$-names $\name{f}^\alpha_n$ of auxiliary functions %$f^\alpha_n$, $\alpha < \zeta$, $n < \omega$, 
    associated with the structure $Y^\alpha$, so that
    $f^\alpha_n: \delta_\alpha \times \finsubsets{\kappa_n^{+(\tau^n_{\ell_n})}} \to \kappa_n^{+(\tau^n_{\ell_n})}$ amalgamates all functions $f \in Y^\alpha$ of the form
    $f : \finsubsets{\kappa_n^{+(\tau^n_{\ell_n})}} \to \kappa_n^{+(\tau^n_{\ell_n})}$, in the sense that these functions appear as $f^\alpha_n \uhr \left(\{ \delta\} \times \finsubsets{\kappa_n^{+(\tau^n_{\ell_n})}} \right)$,  $\delta < \delta_\alpha$. 
    Since $|Y^\alpha| < \zeta$, $\delta_\alpha <\zeta$ as well. 
    The precise form of amalgamation will not be important for us. We only require that the functions naturally extend each other when moving from  $\alpha$ to a bigger $\beta < \zeta$. Namely, for every $n < \omega$, 
    $f^\beta_n \uhr \dom(f^\alpha_n) = f^\alpha_n$.
    We proceed to describe the construction.
    \begin{itemize}
        \item Let $T^0 := T$ and $(t,T) \forces \dot{Y}^0 = \Sk^{\dot{\mathfrak{A}}}(\{\dot{U}_n,\dot{\kappa}_n: n < \omega\})$;
        \item $(t,T^\beta) \dleq (t, T^\alpha)$ for all $\alpha \leq \beta < \zeta$;
        \item $(t,T^\alpha) \forces \dot{Y}^\alpha \prec \dot{\mathfrak{A}}$ for every $\alpha$.
    \end{itemize}
    
    Let us assume we have already constructed $T^\alpha$ and $\dot{Y}^\alpha$. 
   % In $V\left[\vec{\kappa}\right]$, we let $f^\alpha:\delta_\alpha \times \finsubsets{\kappa} \rightarrow \kappa$ be a function such that $\<f \restr (\{\beta\} \times \finsubsets{\kappa}) : \beta < \delta_\alpha\>$ lists all functions $f:\finsubsets{\kappa} \rightarrow \kappa$ in $X^\alpha$. We can and will assume that $f^\alpha$ agrees with previous functions on the intersection of their domains.
    % For $n < \omega$ we let $f^\alpha_n: \finsubsets{\kappa^{+(\tau^n_{k_n})}_n} \rightarrow \kappa^{+(\tau^n_{k_n})}_n$ be the function such that $a \mapsto f^\alpha(a) \mod \kappa^{+(\tau^n_{k_n})}_n$. Pick names $\dot{f}^\alpha_n$ for these objects.
    
    Let $(t,T^{\alpha + 1})$ be a direct extension of $(t,T^\alpha)$ that decides the sequence $\<\dot{f}^\alpha_n : n < \omega\>$ in the sense of 
    Corollary \ref{Cor:PrikryNameGuessing},
    i.e. there is a sequence $\<f^\alpha_s:  s \in T^{\alpha + 1}\>$ such that for every $s \in T^{\alpha + 1}$, we have $(s,T^{\alpha + 1}_s) \forces \dot{f}^\alpha_n = \check{f}^\alpha_s$.
    
    Fix now some $s \in T^{\alpha + 1}$.
    For each $\beta < \delta_\alpha$, define a function 
    $$g^\beta_s: [\pi^{-1}_s(\suc_{T^{\alpha + 1}}(s))]^{<\omega} \to \kappa$$ by 
    $$b \mapsto f^\alpha_{s\concat\<\pi_s(b)\>}(\beta,b).$$
    Note that this definition does not depend on $\alpha$ by the agreement between the $f^\alpha_s$'s. 
    Working with sets in $U_s$ and applying Corollary \ref{genclosedrem},
    we first find a set $A^\alpha_s \in U_s$ so that for each $b \in A^\alpha_s$, $b(i) < (\pi_s(b))^{+\tau^s_i}$, and for each $\beta < \delta_\alpha$ either 
    $$\forall b \in A^\alpha_s.\  g^\beta_s(b) \geq \pi_s(b) \text{ or }$$
   
    $$g^\beta_s \restr A^\alpha_s \text{ is constant. }$$ 
    
    We then shrink each $A^\alpha_s$ to some
    $B^\alpha_s \in U_s$
    so that for each $\beta < \delta_\alpha$, either 
    $$\forall b\in B^\alpha_s. \ g^\beta_s(b) \geq (\pi_s(b))^+  \text{ or }$$ 
  
    $$\forall b\in B^\alpha_s 
    g^\beta_s(b) < h^{\mu_{\beta_s}}(\pi_s(b))$$
    where $h^{\mu_{\beta_s}}$ is one of the canonical functions  fixed at the beginning of the proof. 
    
    %for all $b \in B^\alpha_s$ where $h^\beta_s: \kappa \rightarrow \kappa$ is some canonical function.
    
    We then refine further, using that $U_s$ is minimally generated (Definition \ref{Def:Generators}), $B^\alpha_s$ to $C^\alpha_s \in U_s$ with the property that for all 
    $\beta < \delta_\alpha$,
    $b \in C^\alpha_s$, $i \leq k_s$, and $a \in [b(i)]^{<\omega}$,
    either
    $$f^\alpha_{s\concat\<\pi_s(b)\>}(\beta,a \cup (b \downharpoonright (i + 1)) < b(i), \text{ or }
    $$
 
    $$f^\alpha_{s\concat\<\pi_s(b)\>}(\beta,a \cup (b \downharpoonright (i + 1)) \geq \pi_s(b)^{+\tau^s_i}.
    $$
    
    We then pick $\dot{Y}^{\alpha + 1}$ to be a $\po_U$-name so that $(t,T^{\alpha + 1})$ forces 
    $$\dot{Y}^{\alpha + 1} = \Sk^{\dot{\mathfrak{A}}}\left(\check{\delta_\alpha} \cup \{\dot{Y}^\alpha,\<\check{g}^\beta_s: s \in T^{\alpha + 1}\>,\<\check{h}^{\mu_{\beta_s}} :  s \in T^{\alpha + 1}\>\}\right).$$
    
    At limit stages $\gamma < \zeta$, we take $T^\gamma = \bigcap\limits_{\alpha < \gamma} T^\alpha$ and let $\dot{Y}^\gamma$ be a name for the union of $\dot{Y}^\alpha$ for $\alpha < \gamma$. Similarly, after having constructed the full sequence $\<T^\alpha: \alpha < \zeta\>$ and $\<\dot{Y}^\alpha: \alpha < \zeta\>$ we let $T^\zeta = \bigcap\limits_{\alpha < \zeta} T^\alpha$ and $\dot{Y}$ a name for the union of the $\dot{Y}^\alpha$'s.
    
    This concludes the construction of the two sequences. 
    We make a final direct extension using the sets $C^\alpha_s \in U_s$, $\alpha < \zeta$, that were  introduced in the construction.  
    Let $C_s := \bigcap\limits_{\alpha < \zeta} C^\alpha_s$. We can then find a direct extension $(t,T^*)$ of $(t,T^\zeta)$ with the core property that for every $s \in T^*$, $\suc_{T^*}(s) \subseteq {\pi_s}[{C_s}]$.
    
    Let $\vec{\kappa}$ be a generic sequence with respect to a filter that contains $(t,T^*)$. 
    In $V[\vec{\kappa}]$
    let $Y = \bigcup_{\alpha < \zeta}Y^\alpha$.
    Define sets $C_n := C_{\vec{\kappa} \restr n}$. 
    For each $\kappa_n$ there must be some $b_n \in C_n$ such that $\pi_{\vec{\kappa}\uhr n}(b_n) = \kappa_n$.
    %    We point out that for each $f \in Y^\alpha \subseteq Y$ and  $\beta < \delta_\alpha$, since $g^\beta_s$ is (forced) to be in $X^{\alpha+1}$ and  $b \in A^\alpha_s$
    For each $b_n$ we pick sequences $\<\beta^{n,i}_\gamma: i \leq \ell_n, \gamma < \eta^n_i\>$ such that $\<\beta^{n,i}_\gamma: \gamma < \eta^n_i\>$ is cofinal in $b_n(i)$. 
    We may assume that each sequence $\<\beta^{n,i}_\gamma: \gamma < \eta^n_i\>$ is definable from $b_n(i)$ over $Y$.  We may also require that each $\beta^{n_i}_\alpha$ defines its initial segment $\<\beta^{n,i}_\gamma : \gamma < \alpha\>$ over $Y$. This can be achieved by taking $\beta^{n,i}_\alpha$ to be the rank in our fixed well-order of $\<(\beta^{n,i}_\gamma)': \gamma < \alpha\>$ for a different sequence $\<(\beta^{n,i}_\gamma)': \gamma < \eta^n_i\>$ definable from $b_n(i)$ over $Y$.
    Let $N < \omega$ be minimal such that $\kappa_N \geq \zeta$. Define 
    $$X := \Sk^\mathfrak{A}(Y \cup \{\beta^{n,i}_\gamma: \gamma < \eta^n_i, i \leq \ell_n, N \leq n < \omega\}).$$
    It is clear from the construction of $Y$ that $X$ contains $\zeta$, and therefore satisfies item (i) from the statement of the Lemma. \\
    
    \noindent
    To show that item (ii) holds, we prove that $\sup(X \cap \kappa^{+(\tau^n_i)}_n) = b_n(i) \in S^n_i$ for all $n \geq N$ and $i \leq \ell_n$. It should be clear that the left side is at least $b_n(i)$ by our choice of $\<\beta^{n,i}_\gamma: \gamma < \eta^n_i\>$. 
    Let then $f:\finsubsets{\kappa} \rightarrow \kappa$ be in $Y$ and $r \subset \{\beta^{n,i}_\gamma: \gamma < \eta^n_i, i \leq \ell_n, N \leq n < \omega\}$ such that $f(r) < \kappa^{+(\tau^m_j)}_m$ for some $N \leq m < \omega$ and $j \leq \ell_m$.
    We start with reducing $r$ to a simpler form. First, we may assume that  $r$ does not contain ordinals below $\kappa_m^{+(\tau^m_j - 1)}$, since similarly to the argument of Lemma \ref{thatonelemma}, such ordinals may be suppressed in exchange of replacing $f$ with the function taking $f$ supremums over all ordinals below $\kappa_m^{+(\tau^m_j - 1)}$. 
    Second, we can remove all values from $r$ that are above $\kappa_m^{+(\tau^m_{\ell_m})}$. This follows from the fact that 
    points $\beta \in r \setminus \kappa_m^{+(\tau^m_{\ell_m})}$, are definable from $b_n$ for some $n > m$, and that $b_n$ are taken from the sets $\bigcap_{\alpha < \zeta} A^\alpha_n$. The choice of the sets $A^\alpha_n$ guarantees that $b_n$ end extends $Y$ below $\kappa_m^{+\tau^m_{\ell_m}}$. By Lemma \ref{dblendext}, they also end-extend 
     $\Sk^\mathfrak{A}(Y \cup \{\beta^{k,i}_\gamma: \gamma < \eta^k_i, i \leq \ell_k, k\leq m\})$.
    Having strengthened the assumptions about $r$, it remains to check that for every $j \leq \ell_m$, $r \subset \{\beta^{m,i}_\gamma: \gamma < \eta^n_i, j \leq i \leq \ell_m\}$, and $f \in X$, 
    if $f(r) < \kappa^{+(\tau^m_j)}_m$ then $f(r) < b_m(j)$.
    As $f \in Y$ by construction we have that $f(r) = f_n^\alpha(\delta,r)$ for some $\alpha < \zeta$ and $\delta < \delta_\alpha$. As $b_m \in C_m$ we have $f(r) < b_m(j)$ as wanted. This concludes the verification of item (ii).\\
    
    \noindent 
    Moving to item (iii), we will see that for any  $N \leq m < \omega$, $0 \leq j \leq \ell_m$, and  $\kappa^* \in X \cap (\kappa_m^{+\tau^m_{j-1}},\kappa_m{+\tau^m_{j}})$  is regular,\footnote{Recall that $\tau^m_{-1} = 1$, so the minimal possible $\kappa^*$ value is $\kappa_m^{++}$.} then 
    $\cof(X \cap \kappa^*) \in \{\zeta\} \cup \{ \eta^m_{i} : j \leq i \leq \ell_m\}.$
    Pick such $j$ and $\kappa^*$. By the same reduction argument given above, we have that $$\sup(X \cap \kappa^*) = \sup\left(\kappa^* \cap \Sk^\mathfrak{A} \left(Y \cup \{\beta^{m,i}_\gamma: \gamma < \eta_i^m, i \leq j \leq \ell_m \}\right)\right).$$
    
    For each $s \in \prod\limits_{j \leq i \leq \ell_m} \eta^m_i$, let  $X^s_\gamma := \Sk^\mathfrak{A}(Y \cup \{\beta^{m,i}_{s(i)}: i \leq j \leq \ell_m\})$. It is a limit of an internally approachable sequence of length $\zeta$. Therefore, if for some $s \in \prod\limits_{j \leq i \leq \ell_m} \eta^m_i$, $X^s_\gamma$ is cofinal in $\sup(X \cap \kappa^*)$ then by Lemma \ref{IASkolemhulls},
    $\cof(\sup(X \cap \kappa^*)) = \zeta$. 
    On the other hand, if no
    $X^s_\gamma$ is cofinal in
    $\sup(X \cap \kappa^*)$, then the latter is the supremum of the directed system $\<\sup(X^s_\gamma \cap \kappa^*) : s \in \prod\limits_{j \leq i \leq \ell_m} \eta^m_i\>$ which does not stabilize. Therefore, $\cof(\sup(X \cap \kappa^*)) = \eta^m_i$ for some $j \leq i \leq \ell_m$.\\
    
    \noindent
    We finally verify item $(iv)$, by splitting into two subcases: $\kappa^* \in ( \kappa^{+(\tau^m_{k_m})}_m, \kappa_{m}]$ and $\kappa^* = \kappa_{m}^+$.
    For the first sub-case, the end extension property of the points in $b_n$, $n \geq m+1$ implies that $\kappa^* \in \Sk^\mathfrak{A}(Y \cup \{\beta^{n,i}_\gamma: \gamma < \eta^n_i, i \leq \ell_n , N \leq n \leq m\})$. Therefore $\cof(X \cap \kappa^*) = \zeta$ by Lemma \ref{finitevariationII}.
    For the second sub-case, 
    we note that if  $\gamma \in X$ is less than $\kappa^+_m$, then it is of the form $g^\beta_{\vec{\kappa} \restr n}(b_n)$ for some $\beta < \zeta$. 
    By choice of $b_n$ to be taken from $\bigcap_{\zeta < \alpha} B_{\vec{\kappa}\uhr n}^\alpha$, 
    and definition of the sets $B_s^\alpha$,  we must then have that 
    $$g^\beta_{\vec{\kappa} \uhr n}(b_n) < 
    h^{\mu_{\beta_{\vec{\kappa}\uhr n}}}_{\vec{\kappa} \uhr n}(\kappa_n).$$
    
    By the construction of the sequence 
    $\la Y^\alpha \mid \alpha < \zeta\ra$, 
    both $h^{\mu_{\beta_{\vec{\kappa}\uhr n}}}$ and $\kappa_n$ belong to $Y$, therefore by Lemma \ref{finitevariationII},
    $\sup(X \cap \kappa_m^+) = \sup(Y \cap \kappa_m^+) \in \cof( \zeta)$. 
    
    %It is here where we make use of Lemma \ref{finitevariationII} to fix those gaps after the fact.
    This concludes the proof of Lemma \ref{Lemma:MainLemmaA}, and in turn, of Proposition \ref{Prop:CapturingMain}.
\end{proof}

\begin{remark}\label{Remark:Dodd-Solidity}
    It should be possible to eliminate the ambiguity for regular $\kappa^* \in \left(\kappa^{+(\tau^m_j)}_m,\kappa^{+(\tau^m_{(j + 1)})}_m\right)$. The only way we see to do this involves making extra assumptions on the filters $U_n$. 
    
    If the filter $U_n$ satisfies Dodd-Solidity, i.e. the $\beta \cup b \downharpoonright (j + 2)$-extender derived from $\iota_{U_n}$ is in the ultrapower for every $\beta < b(j + 1)$, then we can define a function that will send $\beta$ to the least $(\{\beta\} \cup b \downharpoonright (j + 2))$-generator below $\kappa^{+\tau'}$ where $\tau'$ is such that $\kappa^{+\tau'}_m = \kappa'$. This function applied to the $\beta^{n,(j + 1)}_\gamma$'s will then be cofinal in $\sup(Y \cap \kappa')$.
\end{remark}

\section{Rudin-Keisler Capturing with Posets}\label{Section:Main} 
In this section we prove theorems \ref{Theorem:WoodinQuestion} and 
\ref{Theorem:SecondMain}. 
by following the proofs of Proposition \ref{Prop:CapturingMain}, with the additional complication having to additionally force with Levy collapse posets between the generic Prikry points. 
To deal with the new complication, we introduce a modification of the Rudin-Keilser capturing property. 

\begin{definition}
     Let $U$ be a normal ultrafilter on $\kappa$, $\qo$ be a forcing notion of size at most $\kappa^+$ in $\ult(V;U)$, and $q \in \partord$ some condition.
     \begin{itemize}
         \item[$(a)$] $\sigma \in V^\qo$ is a $q$-name for a stationary tuple of type $\<(\tau_0,\eta_0),\ldots,(\tau_n,\eta_n)\>$ for positive ordinals $\tau_i$, $i \leq n$, and regular cardinals $\eta_i$, $i \leq n$,
         iff $q$ forces that $\sigma$ is a stationary tuple of type $\<(\check{\tau}_0,\check{\eta}_0),\ldots,(\check{\tau}_n,\check{\eta}_n)\>$. We will usually write such $\sigma$ as a sequence of names $\<\dot{S}_0,\ldots,\dot{S}_n\>$ for its components.
         \item[$(b)$] A function $f: \kappa \rightarrow H_\kappa$ is a $(U,q)$-sequence of stationary tuples of type $\<(\tau_0,\eta_0),\ldots,(\tau_n,\eta_n)\>$ iff $f(\alpha)$ is a $f_q(\alpha)$-name for a stationary tuple of type $\<(\tau_0,\eta_0),\ldots,(\tau_n,\eta_n)\>$ for $U$-almost all $\alpha$. Here $f_q$ is some function representing $q$ in $\ult(V:U)$.
         \item[$(c)$] A $(U,q)$-sequence $f$ of stationary tuples is trivial of type $\<(\tau_0,\eta_0),\ldots,(\tau_n,\eta_n)\>$ iff $f(\alpha)(i) $ is the standard name for the set $\cof(\eta_i) \cap \alpha^{+\tau_i}$ for all $i \leq n$ and $U$-almost all $\alpha$.
    \end{itemize} 
\end{definition}

\begin{definition}
    Let $U$ be a normal ultrafilter on $\kappa$.
    \begin{itemize}
        \item[$(a)$] Let $\qo$ be a forcing notion of size at most $\kappa^+$ in $\ult(V;U)$, $q\in \qo$, and $f$ be some $(U,q)$-sequence of stationary tuples. $(q,f)$ is generically Rudin-Keisler captured (relative to $U$) iff the set of extensions $q' \leq q$ such that there exists a minimally generated ultrafilter $U \leq_{RK} V$ with the property that $q' \forces \check{[id]}_V \in \prod \iota_V(f)(\kappa)$ is dense below $q$.
        \item[$(b)$] We say full generic Rudin-Keisler capturing holds up to $\tau$ (relative to $U$) iff for any $\qo \subset H_{\kappa^+}$ of size at most $\kappa^+$ in $\ult(V:U)$, any $q \in \partord$ and a $(U,q)$-sequence of stationary tuples $f$ of type $\<(\tau_0,\eta_0),\ldots,(\tau_n,\eta_n)\>$ where $1 < \tau_0$ and $\tau_n \leq \tau$, the pair $(q,f)$ is generically Rudin-Keisler captured (relative to $U$).
    \end{itemize}
\end{definition}

\begin{lemma}\label{GenRKcaptcons}
   For any normal ultrafilter $U$ on $\kappa$, full Rudin-Keisler capturing up to $\tau$ relative to $U$ implies full generic Rudin-Keisler capturing relative to $U$ up to $\tau$.
\end{lemma}
\begin{proof}
Let $\<(\tau_0,\eta_0),\ldots,(\tau_n,\eta_n)\>$ be a type with $1 < \tau_0$ and $\tau_n \leq \tau$. 
Let $\qo \in \Ult(V;U)$ be a poset contained in $H_{\kappa^+}$, $q \in \qo$ and
$f$, a $(U,q)$-sequence of stationary tuples. 
For each $i \leq n$ let $\name{S_i} = j_U(f)(\kappa)(i)$. $\name{S_i}$ is a $\qo$-name of a stationary subset of $\kappa^{+\tau_i}$.

Since each $\tau_i > 1$ and $|\qo|^{\Ult(V;U)} \leq \kappa^+$, there must be some $q' \leq q$ such that the set $S_i^{q'} = \{ \alpha< \kappa^{+\tau_i} \mid q' \Vdash \check{\alpha} \in \name{S_i}\}$ is stationary. 
Taking some minimally generated ultrafilter $V \geq_{RK} U$ which captures $S_i^{q'}$, we have that $q' \Vdash [id]_V \in j_V(f)(\kappa)$. 
\end{proof}

The poset $\qo$ for which Ruidn-Keisler capturing will be used are finite produces of Levy-collapse posets. 

\begin{definition}
    Let $s = \la \rho_0,\dots,\rho_{n}\ra$ be an increasing sequence of regular cardinals and $\vec{\tau} = \la \tau(0) \dots,\tau(n-1)\ra$ a sequence of successor ordinals below $\rho_0$. Define the poset $\qo^{\vec{\tau}}_s$ by 
    $$\qo^{\vec{\tau}}_s = \col(\omega_1,\rho_0^+) \times \prod_{i < n} \col(\rho_i^{+(\tau(i) + 1)},\rho_{i+1}^+).$$
    Therefore, conditions $q \in \qo^{\vec{\tau}}_s$ are finite sequence 
    $q = \la q_i \mid i \leq n\ra$ with 
    $q_0 \in \col(\omega_1,\rho_0^+)$ and for $i > 0$, $q_i \in \col(\rho_{i-1}^{+(\tau(i-1) + 1)},\rho_i^{+})$. 
    %As the sequence $\vec{\tau}$ will be fixed in advance in our arguments, and 
    Whenever there is no risk of confusion, we omit the superscript $\vec{\tau}$ and write $\qo_s$ for $\qo_s^{\vec{\tau}}$.
\end{definition}

It is clear that if a sequence $s'$ end extends $s$ then $\qo_{s'}$ projects to $\qo_s$ with the projection map $q' \in \qo_{s'} \mapsto q'\uhr |s|$. 

\begin{definition}\label{barP_U}
    Let $\vec{\tau} = \la \tau(k) \mid k < \omega\ra$ be a sequence of successor ordinals below $\kappa$. 
    Let $\bar{\po}^{\vec{\tau}}_U$ be the Prikry forcing by $U$ with interleaved collapses.
    Conditions $p \in \bar{\po}_U$ are of the form $p = \la s,q,T,G\ra$ where
    \begin{itemize}
        \item $\la s,T\ra \in \po_U$ (See Definition \ref{Def:P_U}),
        \item $q \in \qo^{\vec{\tau}\uhr|s|}_s$,\footnote{In the case $s = \emptyset$, $\qo_\emptyset$ is the trivial forcing.}
        \item $G : T_{>s} \to V_\kappa$ sends a node $t' = t \fr \la \nu\ra \in T_{>s}$ to $G(t') \in \col(\max(t)^{+(\tau(|t|) + 1)},\nu^+)$. 
        %\begin{itemize}
        %   \item if $t = s' \fr \la \nu\ra \in T_{>s}, 
         %   $$G(s \fr \la \nu\ra) \in   \begin{cases}\Col(\max(s)^{+\tau(|s|)},\nu^+) &\mbox{ if } s \neq \emptyset\\\Col(\omega_1,\nu^+) &\mbox{ otherwise}.\end{cases} $$
            
            %\item $G(s) = if $t \in T$ is of the form $t' \fr \la \nu\ra$ for some $t'$ which 
            %\item for every $t \in T$, $G(t) \in \col($,
            %\item if $t,t' \in T$ and $t'$ end extends $t$ then $G(t')\uhr |t| = G(t)$. 
        %\end{itemize}
        
    \end{itemize}
    
    We define the order relations of $\bar{\po}^{\vec{\tau}}_U$.
    \begin{enumerate}
        \item Given $t \in T$ of the form $t = s \fr \la \nu_1,\dots,\nu_k\ra$, the end-extension of $p$ by $t$, denoted $p \fr t$ is the condition $p \fr t = \la t, q \fr (t;G), T_{>t}, G\uhr  T_{>t}\ra$, where 
        $$q \fr (t;G) := q \fr \la G(s \fr \la \nu_1 \ra),G(s \fr \la \nu_1,\nu_2\ra),\dots,G(s \fr \la\nu_1,\dots,\nu_k\ra)$$ 
    
        \item   A condition $p' = \la s',T',q',G'\ra$ is a direct extension of $p$ (written $p' \leq^* p$) iff it satisfies the following conditions:
        \begin{enumerate}
        \item $(s',T') \geq^* (s,T)$ as conditions in $\po_U$, 
        \item  $q' \leq_{\qo_s} q$, and 
        \item  for every $t' = t \fr \la \nu\ra \in T'$, $G'(t') \leq G(t')$ as conditions in $\col(\max(t)^{(\tau(|t|) + 1)},\nu^+).$
                \end{enumerate}

        \item  We say that $p'$ extends $p$ ($p' \leq p$) if $p'$ is obtained from $p$ by a finite sequence of end extensions and direct extensions. 
    \end{enumerate}
\end{definition}

As with the posets $\qo_s^{\vec{\tau}}$, when the identity of the sequence $\vec{\tau}$ is fixed, and there is no risk of confusion, we  suppress the superscript $\vec{\tau}$ and write $\bar{\po}_U$ for $\bar{\po}^{\vec{\tau}}_U$.

\begin{remark}\label{Remark:P_U-factoring}
    Let $p = \la s,q,T,G\ra \in \bar{\po}_U$.
    The forcing $\bar{\po}_U/p$ naturally breaks into the product $\qo_s \times \po_U/ (p\downharpoonright s)$, 
    where $p\downharpoonright s$ is given by $\la \emptyset,\emptyset,T\downharpoonright s,G\uhr (T\downharpoonright s)\ra$, 
    with $T\downharpoonright s = \{ r \in [\kappa]^{<\omega} \mid s \fr r \in T\}$. 
    This decomposition is useful since 
    the direct extension order of $\bar{\po}_U/(p\downharpoonright s)$ is $\max(s)^+$-closed. 
\end{remark}

The following is a standard extension of the basic properties of $\po_U$ to $\bar{\po}_U$.
 (see  \cite{GitikHB}).

\begin{lemma}${}$
     \begin{enumerate}
         \item (Prikry Property) for every $p \in \bar{\po}_U$ and a statement $\varphi$ in the forcing language, there is a $p^* \leq^* p$ which decides $\varphi$.
         
         \item $\bar{\po}_U$ is $\kappa^{++}$.c.c and does not collapse $\kappa^+$
         
         \item If $H \subseteq \bar{\po}_U$ is a generic filter and 
         $$\vec{\kappa} = \la \kappa_n\ra_n = \bigcup\{ s \mid \exists q,T,G.\ (s,q,T,G) \in G\}$$ then in $V[H]$, all cardinals between $\omega_1$ and $\kappa$ are collapsed, except those in the intervals $[\kappa_n^{++},\kappa_n^{+(\tau(n) + 1)}]$, $n < \omega$. 
      \end{enumerate}
\end{lemma}

\begin{lemma}(Strong Prikry Proper  $\bar{\po}_U$) For every condition $p = \la s,q,T,G\ra \in \bar{\po}_U$ and a dense open set $D \subseteq \bar{\po}_U$ there are $p^* = \la s,q^*,T^*,G^*\ra \leq^* p$ and $n < \omega$ such that for every $t \in T$ of length $|t| = n$, the set of conditions $q' \in \qo_t/(q \fr (t;G))$ for which the extension 
$\la t, q',T^*_{>t},G^*\uhr (T^*_{>t})\ra$ of $p \fr t$ belongs to $D$, is dense open in $\qo_t/(q \fr (t;G))$ 
\end{lemma}

With the strong Prikry property, it is straightforward to derive an analogue of Corollary \ref{Cor:PrikryNameGuessing}.
\begin{corollary}\label{Cor:barP_U-NameGuessing}
Suppose that $p \in \bar{\po}_U$ and $\name{x} =\la \name{x_n} \mid n < \omega\ra$ is a $\bar{\po}_U$-name for a sequence of sets $\name{x_n} \subseteq \kappa_n^{+\tau(n)}$. Then there are $p^* = \la s,q^*,T^*,G^*\ra \leq^* p$ and a function $F : T^* \to V_\kappa$, so that for every $t \in T^*$, $F(t)$ is a $\qo_t$-name and 
$$p^* \fr t \Vdash \name{x}_{|t|} = F(t).$$
\end{corollary}

\begin{proposition}\label{Prop:MainProp}
    Let $U$ be a normal measure on $\kappa$.
    Fix some sequence $\<\tau^*_i : i < \omega\>$ of successor ordinals below $\kappa$. Let $H \subset \bar{\partord}^{\<\tau^*_n: n < \omega\>}_U$ be generic over $V$ with $\<\kappa_n : n < \omega\>$ representing $H$'s Prikry component. If the full Rudin-Keisler capturing holds up to $\sup\limits_{n < \omega} \tau^*_i$, then in $V\left[H\right]$ 
    \begin{enumerate}
    \item  
     all cardinals between $\omega_1$ and $\kappa$ are collapsed, except those in the intervals $[\kappa_n^{++},\kappa_n^{+(\tau^*_n + 1)}]$, $n < \omega$. In particular, $\kappa = \aleph_{\bar{\tau}}$ where $\bar{\tau} = \sum_{n < \omega} (1 + \tau^*_n)$,
    
    \item For every $\eta < \kappa$ and infinite sequence of types $\la \mathrm{t}(n) \ra_{n <\omega}$ such that each $\mathrm{t}(n) = \la (\tau^n_0,\eta^n_0), \dots, (\tau^n_{\ell_n},\eta^n_{\ell_n})\ra$ with $1 < \tau^n_0$, $\tau^n_{\ell_n} \leq \tau^*_n$, and $\cup_{i<\ell_n} \eta^n_i < \eta$, 
    any sequence $\vec{S}= \la \vec{S}(n) \mid n < \omega\ra$  of finite sequences of stationary sets $S(n)  = \la S_n^i \mid i \leq \ell_n\ra$ of types $\mathrm{t}(n)$, respectively, is mutually stationary. 
    
    \end{enumerate}
\end{proposition}
 Fix some measurable cardinal $\kappa$, a normal ultrafilter $U$, and a sequence $\<\tau^*_n: n < \omega\>$ as above. Let $\ddot{S}$ be a $\bar{\partord}_U$-name for a sequence of stationary tuples of types $\<(\dot{\tau}^n_i,\dot{\eta}^n_i) : i \leq \ell_n \>$.

We pick some condition $p = (s,q,T,G)$ that decides $\ddot{S}$ in the sense of Corollary \ref{Cor:barP_U-NameGuessing}, i.e. we have some sequence $\<\dot{S}_t: s \eextend t \in T\>$ such that $q \concat (t;G)$ forces $\dot{S}_s$ to be a stationary tuple of type $\<(\tau^t_0,\eta^t_0),\ldots,(\tau^t_{\ell_n},\eta^t_{\ell_n}\>$. Let $p_t$ be the condition represented by the function $\eta \mapsto q \concat (t \fr \la \eta\ra;G)$ in $\ult(V;U)$. We will have a $(p_t,U)$-sequence of stationary tuples $g_t(\eta) := \dot{S}_{t\concat\<\eta\>}$.

We will be done when we prove the next lemma.

\begin{lemma}\label{MainLemmaB}
   Assume that every $(p_t,g_t)$ is generically Rudin-Keisler captured, then there exists some direct extension $(s,q,T^*,G^*)$ of $(s,q,T,G)$ which forces the following: ``for every $N < \omega$, for every $\eta < \zeta < \dot{\kappa}_N$ there is a stationary set of $X \subset \kappa$ such that 
    \begin{itemize}
        \item $\zeta \subset X$,
        \item $\sup(X \cap \dot{\kappa}_m^{+\dot{\tau}^m_i}) \in (\dot{S}(m))_i$ for all $N \leq m < \omega$ and $i \leq \ell_n$,
        \item $\cof(X \cap \kappa^{+\tau^*}_m) \in \{\zeta,\dot{\eta}^m_i,\ldots, \dot{\eta}^m_{\ell_m}\}$ for all $N \leq m < \omega$, all $i \leq  \ell_n$ and all $\dot{\tau}^m_{i - 1} < \tau^* < \dot{\tau}^m_i$ ($\tau^m_{-1} = 1$),
        \item $\cof(X \cap \kappa^*) = \zeta$ for all regular cardinals $\kappa^* > \zeta$ in $X$ not covered by the above."
    \end{itemize}
\end{lemma}

\begin{proof}
    First, we find a direct extension $p' = (s,q,T',G')$ of $(s,q,T,G)$ such that for every $s \eextend t \in T'$ and densely many $q^*$ below $q \concat (t;G')$ there is some minimally generated $U^{q^*}_s \geq_{RK} U$ such that 
    $$q^* \concat G'(t\concat\<\pi^r_s(b)\>) \forces b \in \prod \dot{S}_t$$ for $U^{q^*}_t$-measure one many $b$ where $\pi^{q^*}_t$ is the projection from $U^{q^*}_t$ onto $U$.
    
    We describe  how to get this property at the root $s$ of the condition tree. The full procedure is then just a level by level induction. Let $\<q_\alpha: \alpha < \max(s)^+\>$ be a full list of all conditions extending $q$.
    For a $U$-measure one set of $\eta < \kappa$, 
    we will recursively define a descending sequence $\<r^\eta_\alpha: \alpha \leq \max(s)^+\>$ of conditions $r^\eta_\alpha \in\col(\max(s)^{+(\tau^*_{\length(t)} + 1)}, \eta^+)$ that extend $G(s \concat\<\eta\>)$.
    At limit stages of the construction we take lower bounds using the fact the relevant forcing is $\max(s)^{++}$-closed.
    
    Let $r^\eta_0 = G(s \concat\<\eta\>)$. Suppose that $r^\eta_\alpha$ has been defined for some $\alpha < \max(s)^+$ and a $U$-measure one set of $\eta < \kappa$. Then $g_s$ is also a $(\left[\eta \mapsto q\concat r^\eta_\alpha\right],U)$-sequence of stationary tuples. By the assumption, the pair $(g_s,\left[\eta \mapsto q\concat r^\eta_\alpha\right])$ is generically Rudin-Keisler captured. 
    We can therefore find a witnessing condition, of the form $q^* \fr r^\eta_{\alpha+1}$ on a $U$-measure one set of $\eta$, so that     
    $q^* \leq q_\alpha$ and a minimally generated $U^{q^*}_s$, and $r^\eta_{\alpha + 1}$,  defined for a $U$-measure one set of $\eta$ with the required property.
Having defined all $\<r^\eta_\alpha: \alpha < \max(s)^+\>$ we let $G'(s\concat\<\eta\>)$ be a lower bound of $\<r^\eta_\alpha: \alpha <\max(s)^+\>$. 
     This is defined for all $\eta$ such that every $r^\eta_\alpha$ is defined which is the intersection of $\xi$-many $U$-measure one sets. We then let $\suc_{T'}(s)$ be that intersection.
    
    Note that $q$ remains unchanged. We shall make another harmless assumption here. If $r_0$ and $r_1$ are compatible conditions then we shall have that $U^{r_0}_t = U^{r_1}_t$. Thus we only have one maximal antichain worth of measures at every node.
    
    To keep our notation tidy we shall assume that $(s,q,T,G)$ had the required property to begin with and dispense with $T'$ and $G'$.
    
    Working for a moment in the generic extension we will let $\<\kappa_n: n < \omega\>$ be the generic sequence. $\<H_n: n < \omega\>$ will be the collapses in between the generic sequence. We will also have $U_n := U^{r_n}_{\<\kappa_0,\ldots,\kappa_n\>}$ where $r_n \in \prod\limits_{i \leq n} H_i$. Let us pick names $\<\dot{\kappa}_n: n < \omega\>$, $\<\dot{H}_n: n < \omega\>$, and $\<\dot{U}_n: n < \omega\>$.
    
    Pick some $\zeta < \kappa$ which remains a regular uncountable cardinal in the extension. We can and shall assume that $\zeta < \max(s)$ otherwise we must extend $s$.
    
    We now build sequences $p^\alpha := \<(s,q,T^\alpha,G^\alpha): \alpha < \zeta\>$ and $\<\ddot{X}^\alpha: \alpha < \zeta\>$ as in the proof of Lemma \ref{Lemma:MainLemmaA}, i.e.
    
    \begin{itemize}
        \item $T^0 := T$, $G^0 := G$ and $(s,q,T,G) \forces \ddot{X}^0 = \Sk^{\dot{\mathfrak{A}}}(\{\dot{U}_n,\dot{\kappa}_n,\dot{G}_n: n < \omega\})$;
        \item $(s,q,T^\beta,G^\beta) \dleq (s,q, T^\alpha,G^\alpha)$ for $\alpha \leq \beta < \zeta$;
        \item $(s,q,T^\alpha,G^\alpha) \forces \ddot{X}^\alpha \prec \dot{\mathfrak{A}}$.
    \end{itemize}
    
    Let us assume $T^\alpha$, $G^\alpha$, and $\ddot{X}_\alpha$ have already been constructed. Let $\ddot{f}^\alpha_n$ be names for functions amalgamating the functions $f:\finsubsets{\kappa^{+(\tau^*_n)}_n} \rightarrow \kappa^{+(\tau^*_n)}_n$ in $\ddot{X}_\alpha$. 
    
    We let $p' := (s,q,T',G') \dleq (s,q,T^\alpha,G^\alpha)$ that decides the sequence $\<\ddot{f}^\alpha_n: n < \omega\>$, i.e there is a sequence $\<\dot{f}^\alpha_s: s \eextend t \in T'\>$ of collapse-product names such that $p' \concat t \forces \dot{f}^\alpha_s = \ddot{f}^\alpha_{\length(s)}$.
    
    We then define (partial) functions $f^\alpha_t$ such that $f^\alpha_t(r,\beta,a) = \gamma$ iff $r \forces \dot{f}^\alpha_s(\check{\beta},\check{a}) = \check{\gamma}$ where $r \leq q \concat (t;G')$.
    
    We will now define $(s,q,T^{\alpha + 1},G^{\alpha + 1})$ in a similar construction to the start of this proof. They must have the property that for every $s \eextend t$ in $T^{\alpha + 1}$ and densely many $r \leq q \concat (t;G')$, there is some $B^{\alpha,r}_t \in U^r_t$ such that for all $\beta < \delta_\alpha$, 
    defining  
    $$g^\beta_{t,r}(b) = f^\alpha_{t\concat\<\pi^r_t(b)\>}(r\concat G^{\alpha + 1}(t\concat\<\pi^r_t(b)\>),\beta,b)$$ then 
    either
    \begin{itemize}
        \item $g^\beta_{t,r}(b) \geq \pi^r_t(b)$, or 
        \item $g^\beta_{t,r}(b)$ is constant on $B^{\alpha,r}_t$.
    \end{itemize}    
    
    As before we will only detail the procedure for $t = s$. Let $\<q_\gamma: \gamma < \xi\>$ be a full list of all conditions extending $q$. We will recursively define a descending list of $\<r^\eta_\gamma: \gamma < \xi\>$ of conditions $r^\eta_\gamma \leq G'(s \concat\<\eta\>)$ where $\eta$ is taken from some $U$-measure one set.
    
    If $r^\eta_\gamma$ is already given, then there is a $U^{q_\gamma}_s$-measure one set $A^{q_\gamma}_s$ such that for every $\beta < \delta_\alpha$, for every $b \in A^{q_\gamma}_s$ there is some $q^* \leq q_\gamma$ and $r^*_b \leq r^{\pi^{q_\gamma}_t(b)}_\gamma$ such that $q^* \fr r^*_b$ decides the value of $\dot{f}^\alpha_{s\concat\<\pi^q_s(b)\>}(\check{\beta},\check{b})$. We let $r^\eta_{\gamma + 1} = r^*_b$ for $b \in A^{q_\gamma}_s$ with $\pi^{q_\gamma}_s(b) = \eta$.
    
    Finally, we then let $G^{\alpha + 1}(s\concat\<\eta\>)$ be some lower bound of $\<r^\eta_\gamma: \gamma < \xi\>$ where $\eta \in \suc_{T'}(s) \cap \bigcap\limits_{\gamma < \xi} {\pi^{q_\alpha}_s}[A^{q_\alpha}_s]$. Once again we let $G^{\alpha + 1}(s) = G'(s)$. Similarly, the newly defined $G^{\alpha + 1}(s\concat\<\eta\>)$ will not be subjected to further changes throughout the construction.
    
    We then have that $g^\beta_{s,r}$ as above is well-defined on a $U^r_s$-measure one set for a dense set of $r$'s. We can then find sets $B^\alpha_{s,r}$ with the required property.
    
    We refine $B^\alpha_{t,r}$ further to a set $C^\alpha_{t,r}$ such that for every $b \in C^\alpha_{t,r}$, for every $\beta < \delta_\alpha$, for every $i \leq \ell_{|t|}$, every $a \subset b(i)$ and 
    $r^* \leq r\concat G^{\alpha + 1}(\pi^r_t(b))$ if 
    $$f^\alpha_{t\concat\<\pi^r_t(b)\>}(r^*, \beta, a \cup (b \downharpoonright (i + 1))) < \pi^q_s(b)^{+\tau_i},$$ then 
    $$f^\alpha_{t\concat\<\pi^r_t(b)\>}(r^*, \beta, a \cup (b \downharpoonright (i + 1))) < b_i.$$
    
    We then let $\ddot{X}^{\alpha + 1}$ be a name such that $p^{\alpha + 1} \forces \ddot{X}^{\alpha + 1} = \Sk^{\dot{\mathfrak{A}}}(\check{\delta}_\alpha \cup \{\ddot{X}^\alpha,\<\check{g}^\beta_{t,r}: s \eextend t \in T^{\alpha + 1}, r \leq q \concat (t;G^{\alpha + 1})\>\})$.

    At limit stages $\gamma < \zeta$, we let $T^\gamma = \bigcap\limits_{\alpha < \gamma} T^\alpha$, $G^\gamma(t)$ a lower bound for $\<G^\alpha(t): \alpha <\gamma\>$, and let $\ddot{X}^\gamma$ be a name for the union of $\ddot{X}^\alpha$ for $\alpha < \gamma$. Here it is important that we assumed $\zeta < \max(s)$ so that we have enough closure in the relevant collapse parts.
    
    Similarly, after having constructed the full sequence $\<T^\alpha: \alpha < \zeta\>$ and $\<\ddot{X}^\alpha: \alpha < \zeta\>$ we let $T^\zeta = \bigcap\limits_{\alpha < \zeta} T^\alpha$, $G^\zeta$ a lower bound for $\<G^\alpha: \alpha < \zeta\>$, and $\ddot{X}$ a name for the union of the $\dot{X}^\alpha$'s.
    
    In the process we have also constructed $C^r_t := \bigcap\limits_{\alpha < \zeta} C^\alpha_{r,t}$. We find the final condition $(t,T^*,G^*)$ by shrinking the tree such that $\suc_{T^*}(t) \subset \bigcap \ptwimg{\pi^r_t}{C^r_t}$.
    
    In the extension we then have $C_n \in U_n$. For all $n > \length(s)$ we will have some $b_n \in C_n$ that projects down to $\kappa_n$. We pick then suitable $\<\beta^{i,n}_\gamma: \gamma < \eta^i_n\>$ cofinal in $b_n(i)$. The final structure then is $Y:= \Sk^\mathfrak{A}(X \cup \{\beta^{i,n}_\gamma: \gamma < \eta^n_i, \length(s) < n < \omega, i \leq \ell_n\})$.
    
    The argument to check that this works is almost the same as in Lemma \ref{Lemma:MainLemmaA}. The only addition is that, here, for every $\alpha < \zeta$ we have an antichain of values for $f^{\alpha'}_n(\alpha,b)$ given by the functions $g^{r,t}_\alpha$. This is why we put the sequence $\<H_n: n < \omega\>$ into the hull as it will tell us which one gives the correct value. The rest is as before.
\end{proof}

We record a version of the last proposition, applied to trivial types. Its proof is a straightforward modification of the last.
\begin{corollary}
        Let $U$ be a normal measure on $\kappa$ and let $\<k_n: n < \omega\>$ be a sequence of positive natural numbers. Let $H \subset \bar{\partord}^{\la 2k_n - 1\ra_{n < \omega}}_U$ generic over $V$ with $\<\kappa_n : n < \omega\>$ representing its Prikry part. If the trivial sequence of type $\<(2,\eta),\ldots,((k + 2),\eta)\>$ is Rudin-Keisler captured relative to $U$, then in $V[G]$ $\kappa = \aleph_\omega$ and for every $N < \omega$ there is a stationary set of $X \subset \kappa$ such that 
        \begin{itemize}
            \item $\kappa^{+(2(k + 1) + 1)}_N \subset X$,
            \item $\cof(X \cap \kappa^{+(l + 2)}_n) = \eta$ for every $N < n < \omega$ and $l \leq k$,
            \item $\cof(X \cap \kappa^{+(l + (k + 3))}_n) = \zeta$ for every $N < n < \omega$ and $l \leq k$.
            \end{itemize}
\end{corollary}

With Proposition \ref{Prop:MainProp} we can prove the main results stated in the introduction. Starting from Theorem \ref{Theorem:WoodinQuestion}.

\begin{proof}(Theorem \ref{Theorem:WoodinQuestion})\\
Suppose that $\kappa$ is $(\omega+2)$-strong in $V$ and
$j : V \to M$ be a witnessing elementary emebdding. Denote its induced normal measure by $U$.
By Lemma \ref{RKcaptcons}, the full Rudin-Keisler capturing holds up to $\tau = \omega+1$, with respect to $U$, and by
 Lemma \ref{GenRKcaptcons}, the full generic Rudin-Keisler capturing holds up to $\tau = \omega+1$. 
Let $\bar{\po}_U = \bar{\po}_U^{\vec{\tau}}$ for the constant sequence $\vec{\tau} = \la \tau_n^* \mid n < \omega\ra$, of $\tau^*_n = \omega+1$ for all $n < \omega$, and
 $H \subseteq \bar{\po}_U$ be a generic filter. Denote its Prikry sequence by $\la \kappa_n \mid n < \omega\ra$.
 By Proposition \ref{Prop:MainProp}, in $V[H]$, all cardinals below $\kappa$ are collapsed except those in the intervals $[\kappa_n^{++},\kappa_n^{+\omega+2}]$. It follows that $\kappa = \aleph_{\omega^2}$ and for each $n$, $\aleph_{\omega n + 1} = \kappa_n^{+\omega+1}$. Moreover, for every  $V[H]$ cardinal $\eta < \kappa$, considering the simple length-one types $t(n) = \la(\omega+1,\eta)\ra$ for all $n$, the proposition asserts that  every sequence in $V[H]$, $\vec{S} = \la S_n^0 \mid n < \omega, \eta < \kappa_n\ra$,  of stationary sets $S_n^0 \subseteq \kappa_n^{+\omega+1} \cap \cof(\eta)$, is mutually stationary. 
\end{proof}

We turn to the second main theorem. 
\begin{proof}(Theorem \ref{Theorem:SecondMain})\\
Suppose that $\kappa$ is $\omega$-strong in $V$, and $U$ a normal measure on $\kappa$ derived from a witnessing embedding.
By Lemmas \ref{RKcaptcons} and \ref{GenRKcaptcons}
the full generic Rudin-Keisler capturing up to $\omega$ holds for $U$. Let $\bar{\po}_U = \bar{\po}_U^{\vec{\tau}}$ where $\vec{\tau} = \la \tau_n^* \mid n < \omega\ra$ is given by $\tau_n^* = k_n$.
The result follows at once from Proposition \ref{Prop:MainProp}.
\end{proof}

\section{Towards an optimal assumption}\label{Section:WeakRKcapturing}

\begin{definition}
    Let $U$ be a normal ultrafilter on $\kappa$. Let $\<(\tau_0,\eta_0),\ldots,(\tau_n,\eta_n)\>$ be a sequence of pairs such that $\tau_i,\eta_i < \kappa$ for all $i \leq n$.
    \begin{itemize}
        \item[$(a)$] A Rudin-Keisler system of type $\<(\tau_0,\eta_0),\ldots,(\tau_n,\eta_n)\>$ (over $U$) is a sequence $\<U_s: s \in \prod\limits_{i \leq n} \eta_i\>$  such that
        \begin{itemize}
            \item $U_s \geq_{RK} U$ is such that $a_s(i) < (\kappa^{+\tau_i})^{\ult(V;U_s)}$ where $a_s$ is the principal element of $U_s$;
            \item $U_t \geq_{RK} U_s$ (as witnessed by $\pi_{t,s}$) whenever $t,s$ are such that $t(i) \geq s(i)$ for all $i \leq n$;
            \item whenever $t,s$ are such that $t \downharpoonright (i + 1) = s \downharpoonright (i + 1)$, $t(i) > s(i)$, and $t(j) \geq s(j)$ for $j < i$, then $\iota_{U_t}(f)(\iota_{U_t}(\pi_{t,s})(a_t) \downharpoonright i) > (\kappa^{+\tau_i})^{\ult(V;U_t)}$ or $\iota_{U_t}(f)(\iota_{U_t}(\pi_{t,s})(a_t) \downharpoonright i) < a_t(i)$ for all $f:\finsubsets{\kappa} \rightarrow \kappa$.
        \end{itemize}
        \item[$(b)$] Let $\<U_s: s \in \prod\limits_{i \leq n} \eta_i\>$ be a Rudin-Keisler system over $U$ as witnessed by $\<\pi_{t,s}\>$ and $\<\pi_s\>$. Let $A_s \in U_s$. A run through $\<A_s: s \in \prod\limits_{i \leq n} \eta_i\>$ with base $\alpha$ is a tuple $\<\<\alpha^i_\gamma: \gamma < \eta_i\>: i \leq n\>$ such that
        \begin{itemize}
            \item $\vec{\alpha}_s := \<\alpha^i_{s(i)}: i \leq n\> \in A_s$ for all $s \in \prod\limits_{i \leq n} \eta_i$;
            \item $\pi_{t,s}(\vec{\alpha}_t) = \vec{\alpha}_s$ for all $t,s \in \prod\limits_{i \leq n} \eta_i$ such that $t(i) > s(i)$ for all $i \leq n$;
            \item $\pi_s(\vec{\alpha}_s) = \alpha$ for all $s \in \prod\limits_{i \leq n} \eta_i$.
        \end{itemize}
    \end{itemize} 
\end{definition}

\begin{definition}
    Let $U$ be a normal ultrafilter on $\kappa$. Let $f$ be a $U$-sequence of stationary tuples of type $\<(\tau_0,\eta_0),\ldots, (\tau_n,\eta_n)\>$. $f$ is weakly Rudin-Keisler captured iff there exists a Rudin-Keisler system of type $\<(\tau_0,\eta_0),\ldots,(\tau_n,\eta_n)\>$ $\<U_s: s \in \prod\limits_{i \leq n} \eta_i\>$ such that for every $\<A_s: s \in \prod\limits_{i \leq n} \eta_i\>$ with $A_s \in U_s$ there are $U$-measure many $\alpha < \kappa$ such that there exists a run $\<\<\alpha^i_\gamma: \gamma < \eta_i\>: i \leq n\>$ through $\<A_s: s \in \prod\limits_{i \leq n}\>$ with base $\alpha$ and the additional property that $\sup\limits_{\gamma < \eta_i} \alpha^i_\gamma \in f(\alpha)(i)$ for all $i \leq n$.
\end{definition}

\begin{lemma}
     Let $U$ be a normal ultrafilter on $\kappa$. Let $f$ be a $U$-sequence of stationary tuples. If $f$ is Rudin-Keisler captured, then it is weakly Rudin-Keisler captured.
\end{lemma}

\begin{proof}
    Let $\<(\tau_0,\eta_0),\ldots,(\tau_n,\eta_n)\>$ be the type of $f$, and let $U_f$ witness that $f$ is Rudin-Keisler captured. In $\ult(V;U_f)$ we then have the principal element $a$ of $U_f$ with the property that $a(i) \in \iota_{U_f}(f)(\kappa)(i)$. We pick then in the ultrapower sequences $\<\alpha^i_\gamma: \gamma < \eta_i\>$ for $i \leq n$ such that 
    \begin{itemize}
        \item $\sup\limits_{\gamma < \eta_i} \alpha^i_\gamma = a(i)$,
        \item $\iota(\pi^i_{\gamma,\delta})(\alpha^i_\gamma) = \alpha^i_\delta$ for some $\pi^i_{\gamma,\delta}: \kappa \rightarrow \kappa$ for all $\delta < \gamma < \eta_i$ and $i \leq n$.
    \end{itemize}
    
    (We do not know how to pick $\alpha^i_\gamma$ with these properties while also maintaining that they are generators. This is why we did not require elements of a Rudin-Keisler system to be minimally generated. Never the less we can assume that there is always at least one generator between any two elements in the sequence which is sufficient.)
    
    Define then the Rudin-Keisler system by letting $U_s := U^{\iota_{U_f}}_{\vec{\alpha}_s}$ where $\vec{\alpha}_s = \<\alpha^i_{s(i)}: i \leq n\>$.
    
    For any $\<A_s: s \in \prod\limits_{i \leq n} \eta_i\>$ with $A_s \in U_s$ the $\alpha^i_\gamma$'s then constitute a run through $\<\iota_{U_f}(A_s): s \in \prod\limits_{i \leq n} \eta_i\>$ with base $\kappa$ such that $\sup\limits_{\gamma < \eta_i} \alpha^i_\gamma \in \iota_{U_f}(f)(\kappa)(i)$ witnessing that $f$ is weakly Rudin-Keisler captured by $\<U_s: s \in \prod\limits_{i \leq n} \eta_i\>$.
\end{proof}

\begin{remark}
    The reverse implication does not hold. Working in a canonical inner model we can use the coherence  of the extender sequence to show that in the above argument the Rudin-Keisler system is an element of the ultrapower $\ult(V;U_f)$, and by an absoluteness argument also witnesses weak Rudin-Keisler capturing there. By a reflection argument we then have that $\kappa$ is already a limit of cardinals where a sequence of that type is weakly Rudin-Keisler captured.
\end{remark}

We want to close by pointing out how the weak Rudin-Keisler capturing is sufficient for all the arguments we have made thus far. 

Considering for example the proof of Proposition \ref{Prop:CapturingMain}, instead of working with Filters $U_s$ for every node in the tree end-extending the stem we will have a Rudin-Keisler system $\<U^s_t: t \in \prod\limits_{i \leq \ell_s} \eta^s_i\>$. 

In the end we will have sets $C^s_t \in U^s_t$ of ordinals sufficiently closed under the Skolem functions of the hull we were building. We then choose the final condition to be such that for every splitting node $s$ every successor has a run through $\<C^s_t: t \in \prod\limits_{i \leq \ell_s} \eta^s_i\>$ with it as its base and takeing supremums up to be in the right set.

Working in the generic extension we then have that $\kappa_n$ has a run of sufficiently closed ordinals with it as its base for all but finitely many $n$. Those runs can then be used to end-extend the structure in the correct fashion.

\section{Open Problems}\label{Section:OpenProblem}

Much remains unknown regarding the ability to produce special elementary substructures of singular cardinals. The most well-known problem around this topic is whether $\aleph_\omega$ can be a Jonsson cardinal, but there are many other configurations that remain unknown. 
The notion of Rudin-Keisler capturing and its use suggests that configurations of measures and generators could lead to new developments. We list a number of related problems. \\

The notion of Rudin-Keisler capturing can be associated with extenders, but as opposed to extenders, the the measures used in the Rudin-Keisler arguments do not have to cohere themselves but only project to the same normal measure. Nevertheless, our examples of Rudin-Keisler capturing are derived from strong-embeddings whose finer properties seem to have an effect on the some of the outcomes (For example,  Remark \ref{Remark:Dodd-Solidity} points out that cofinality of the structure constructed in Proposition \ref{Prop:CapturingMain} below certain cardinals could be effected by properties such as Dodd-Solidity of an underlying extender).
There are many types of Rudin-Keisler systems with different behaviours, such as coming from extenders with long generators, weak extenders which witness a cardinal being tall but not strong, or  extenders comprising stronger types of measures such as strongly compact measures and various regular or irregular measures on different cardinals. 
%\begin{question}Are there ``exotic'' types of Rudin-Keisler capturing systems that produce different mutually stationary behaviour\end{question}
Specifically, one can ask about the ability to produce Chang-type structures.
\begin{question}
Can one obtain various Chang-type structure using RK-capturing arguments with strong measures?
\end{question}

Another perspective that could lead to ``exotic'' types of Rudin-Keisler capturing systems is from forcing that produce special types of scales, such as in \cite{10.2307/20535089}, \cite{cummings_foreman_2010}, and \cite{Sinapova2015}.
It is also natural to ask if extenders are actually necessary to construct Rudin-Keisler capturing systems. 
We propose the result of Theorem \ref{Theorem:WoodinQuestion}, which uses an $\omega+2$-strong cardinal, as a test question.

\begin{question}
What is the consistency strength of the conclusion in Theorem \ref{Theorem:WoodinQuestion}?
\end{question}

The argument in section \ref{Section:WeakRKcapturing} suggest that $(\omega+2)$-strong is not an optimal assumption.\footnote{In fact,  our mutual stationarity results are obtained by adding a Prikry sequence with respect to a normal measure $U$, and it is not hard to see that all are satisfied in ultrapowers of $V$ which see $U$. Therefore,  the optimal assumptions should only use a normal measure concentrating on a certain type of ordinals.} 
On the other hand, very weak Rudin-Keisler systems, such as the generically induced systems used in \cite{HOD2} and
\cite{AdolfBN} seem to be too weak for carrying the Rudin-Keisler capturing arguments. 
In \cite{mutstat}, the authors introduce a stronger notion of tight stationarity. Consistency results about tightly stationary sets show connections to extender based forcing (\cite{singularstatII}) and PCF theory (\cite{LiuShe-MS}).
Much remains unknown about the consistency results of tightly stationary sequences. 

\begin{question}
Is it consistent that for some sequence $\la \kappa_n \mid n < \omega\ra$ of regular cardinals and some cofinality $\rho < \kappa_0$, all sequences of stationary sets $S_n \subseteq \kappa_n \cap \cof(\rho)$ are tightly stationary?
\end{question}

%Their arguments suggest that tight stationarity is closely related to forcing with extender-based Prikry forcings, which allow to identify sequences of indiscrnibles in a product $\prod_n \kappa_n^{+\tau_n}$ with a single ordinal above $\sup_n \kappa_n^{+\tau_n}$ given by an index of a generically added scale function. 

\begin{question}
Is there an extender-like forcing for Rudin-Keisler capturing systems, which can make every mutually stationary sequence in the models of Theorems \ref{Theorem:WoodinQuestion} or \ref{Theorem:SecondMain} become tightly stationary?
\end{question}

%Add an open problem section at the end, including (1) Is it possible to using the RK-capturing method to cover all sequences of stationary sets (say below aleph_omega) [maybe with very long extenders]\\
%(2) Is it possible to prove similar result for uncountable sequences (replacing Prikry forcing with Magidor), \\
%(3) Is is possible to get a chang-type model using these methods?\\
%(4) Can we apply these methods to a chang-type model and get intersting conclusions?
%(5) What are the optimal assumptions for our theorems (something like a normal measure concentrating on points with property X)
%(6) What about Tight stationarity?

Finally, we ask about extending the Rudin-Keisler capturing-based results to longer sequences of stationary sets.

\begin{question} (Mutual Stationarity up to $\aleph_{\omega^2}$)
Is it consistent that for every $\beta < \omega^2$, every sequence $\la S_{\alpha+1} \mid \beta \leq \alpha < \omega^2\ra$ of stationary sets $S_{\alpha+1} \subseteq \aleph_{\alpha+1} \cap \cof(\leq\aleph_\beta)$ is mutually stationary?
\end{question}

\begin{question}
Can the Rudin-Keisler capturing machinery work for uncountable sequences of stationary sets?  
\end{question}

\subsection*{Acknowledgements}
The authors would like to thank Yair Hayut and Menachem Magidor for valuable comments and suggestions regarding this work.

\bibliographystyle{plain}
\bibliography{references}
\end{document}

%% file: DominikMakros.tex
\newcommand{\<}{\langle}
\renewcommand{\>}{\rangle}
\newcommand{\op}[1]{\operatorname{#1}}
\newcommand{\restr}{\upharpoonright}
\newcommand{\ult}{\op{Ult}}

\newcommand{\cof}{\op{cof}}
\newcommand{\card}{\op{card}}

\newcommand{\ptwimg}[2]{{#1}"\left[{#2}\right]}
\newcommand{\kleiner}{\mathord{<}}
\newcommand{\kleinergleich}{\mathord{\leq}}

\newcommand{\length}{\op{lh}}
\newcommand{\eextend}{\trianglelefteq}

\newcommand{\concat}{{}^\smallfrown}

\newcommand{\finsubsets}[1]{{\left[#1\right]}^{\kleiner\omega}}

\newcommand{\Sk}{\op{Sk}}
\newcommand{\Pot}{\op{\mathcal{P}}}
\newcommand{\partord}{\mathbb{P}}
\newcommand{\forces}{\Vdash}
\newcommand{\col}{\op{col}}
\newcommand{\dleq}{\leq^*}
\newcommand{\suc}{\op{suc}}

%% file: OmerMakros.tex
\newcommand{\po}{\mathbb{P}}
\newcommand{\qo}{\mathbb{Q}}
\newcommand{\la}{\langle}
\newcommand{\ra}{\rangle}
\newcommand{\fr}{{}^{\frown}}
\newcommand{\name}{\dot}
\newcommand{\elem}{\prec}
\newcommand{\uhr}{\upharpoonright}

\newcommand{\power}{\mathcal{P}}
\newcommand{\A}{\mathfrak{A}}

\DeclareMathOperator{\dom}{dom}

\DeclareMathOperator{\cp}{cp}
\DeclareMathOperator{\Ult}{ult}